\documentclass[a4paper]{article}

\usepackage[utf8]{inputenc}
\usepackage{lmodern}
\usepackage[T1]{fontenc}

\usepackage{mathtools, amssymb, amsthm}
\usepackage{enumitem}
\usepackage{bbm,dsfont}
\numberwithin{equation}{section}

\newcommand{\K}{\mathbb{K}}
\newcommand{\R}{\mathbb{R}}
\newcommand{\C}{\mathbb{C}}

\newcommand{\A}{\mathcal{A}}

\renewcommand{\L}{\mathcal{L}}
\newcommand{\B}{\mathcal{B}}
\renewcommand{\d}{\mathrm{d}}
\renewcommand{\Re}{\operatorname{Re}}

\newcommand{\fra}{\mathfrak{a}}

\renewcommand{\mid}{\, \vert \,}
\newcommand{\MR}{\textit{MR}\,}
\DeclarePairedDelimiter\abs{\lvert}{\rvert}
 \DeclarePairedDelimiter\norm{\lVert}{\rVert}
\newcommand{\z}{\zeta}
\renewcommand{\H}{\mathcal{H}}

\theoremstyle{plain}
\newtheorem{theorem}{Theorem}[section]
\newtheorem{proposition}[theorem]{Proposition}
\newtheorem{lemma}[theorem]{Lemma}
\newtheorem{corollary}[theorem]{Corollary}
\theoremstyle{definition}
\newtheorem{definition}[theorem]{Definition}

\newtheorem{remark}[theorem]{Remark}
\newtheorem{assumption}[theorem]{Assumption}
\begin{document}
\title{ On the right multiplicative perturbation of non-autonomous $L^p$-maximal regularity\footnote {Work  supported by Deutsche Forschungsgemeinschaft  DFG (Grant JA 735/8-1)}}
\author{
Bj\"orn Augner, Birgit Jacob, Hafida Laasri
}


\maketitle

\begin{abstract}\label{abstract}
This paper is devoted to the study of $L^p$-maximal regularity for non-autonomous linear evolution equations of the form
\begin{equation*}\label{Multi-pert1-diss-non}
  \dot u(t)+A(t)B(t)u(t)=f(t)\ \ t\in[0,T],\ \
 u(0)=u_0.
  \end{equation*}
where $\{A(t),\ t\in [0,T]\}$ is a family of linear unbounded operators whereas the operators $\{B(t),\ t\in [0,T]\}$ are bounded and invertible. In the Hilbert space situation we consider  operators  $A(t), \ t\in[0,T],$ which arise from  sesquilinear forms. The obtained results are applied to parabolic linear differential equations in one spatial dimension.
\end{abstract}

\medskip


\section{Introduction\label{s2}}
We consider the following partial differential equation
\begin{align}\label{dissipative system}
  \frac{\partial u}{\partial t}(t,\zeta)-\frac{\partial}{\partial \zeta}\big(GS\frac{\partial}{\partial \zeta}G^*\mathcal{H}u+P_1\mathcal{H}u\big)(t,\zeta)
  -P_0(\mathcal{H}u)&(t,\zeta)=f(t,\zeta),
\\&\nonumber \zeta\in[0,1], t\geq 0,
  \end{align}
where $S\in L^{\infty}(0,1;\mathbb C^{k\times k})$ and $\H\in L^{\infty}(0,1;\mathbb C^{n\times n})$ are  coercive  multiplication operators on $L^2(0,1;\mathbb C^{k})$ and $L^2(0,1;\mathbb C^{n}),$ respectively,  $G\in \mathbb C^{n\times k}$ and  $P_1, P_0\in L^{\infty}(0,1; \mathbb C^{n\times n}).$  
We write (\ref{dissipative system}) as the abstract Cauchy problem
\begin{equation}\label{Multi-pert1-diss}
  \dot u(t)+A\mathcal{H}u(t)=f(t),\ \
 u(0)=0
  \end{equation}
where the operator $A$ is given by
\begin{equation}\label{ellip+ope}
A:=-\frac{\partial}{\partial \zeta}\left(GS\frac{\partial}{\partial \zeta}G^*+P_1\right)-P_0
\end{equation}
on a domain $D(A)$ which includes appropriate boundary conditions. We aim to characterize boundary conditions such that  $-A\mathcal H$ with
domain $\{ u\in L^2(0,1;\mathbb C^n):\ \mathcal Hu\in D(A)\}$ generates a holomorphic $C_0$-semigroup on $L^2(0,1;\mathbb C^n).$ Furthermore, we investigate whether $-A\mathcal H$ generates a holomorphic $C_0$-semigroup if and only if $-A$ generates a holomorphic $C_0$-semigroup.
We remark, that in \cite[Chapter 6]{Vi}, \cite{Vi-Go-Zw-Ma} (see also \cite{Mi-Zw}) closure relation methods are used to show that $-A\mathcal H$ generates a contraction semigroup for suitable boundary conditions.
\par\noindent  If $S$ and $\mathcal H$ also  depend  on the  time variable $t\in[0,T] ,$  problem (\ref{Multi-pert1-diss}) becomes  a non-autonomous Cauchy problem
\begin{equation}\label{Multi-pert1-diss-non}
  \dot u(t)+A(t)\mathcal{H}(t)u(t)=f(t)\ \ t\in[0,T],\ \
 u(0)=0.
  \end{equation}
We are interested in the well-posedness of (\ref{Multi-pert1-diss-non}) with $L^p$-maximal regularity. Again, as in the autonomous case, it is natural to ask whether  well-posedness  of (\ref{Multi-pert1-diss-non})  with $\mathcal H(t)=I$ implies  well-posedness in the general case.
\par\noindent  Motivated by this example, we start a systematic study of stability of $L^p$-maximal regularity under multiplicative perturbation in a more general situation. First, in Section \ref{MR-Banach spaces} we study  $L^p$-maximal regularity $(p\in(1,\infty))$ for  non-autonomous evolutionary linear Cauchy problems of the form
\begin{equation}\label{Multi-pert1-form}
  \dot u(t)+A(t){B}(t)u(t)=f(t) \ \hbox{a.e.\   on } (0,T),\ \ B(0)u(0)=x_0.
\end{equation}
Here $A:[0,T]\longrightarrow \L(D,X)$ is a strongly measurable function, where $D$ and $X$
are two Banach spaces such that $D\underset{d}{\hookrightarrow}X$, the space $X$ has  the Radon-Nikodým
property and  $B:[0,T]\longrightarrow \L(X).$ Note  that although the domains of the operators $A(t)$ are constant the domains of the perturbed operator $A(t)B(t)$
\[D(A(t)B(t)):=\{ u\in H: B(t)u\in D\}\] may depend on the time variable $t.$
In comparison to the autonomous case, $L^p$-maximal regularity for evolution equations related to
non-autonomous operator families $\{C(t),\ t\in (0,T)\}$ is less well understood.
However, several results have been established. We will mention some of them, distinguishing
between the case where all the operators $C(t)$ have the same domain and the more general
case of time-dependent $D(C(t)).$ In the latter situation Hieber and Monniaux \cite{Hi-Mo1, Hi-Mo2} and Portal and Strkalj \cite{Po-St}  proved $L^p$-maximal regularity, if all operators $C(t)$ have the $L^p$-maximal regularity and the family $\{C(t),\ t\in (0,T)\}$
 satisfies the Acquistapace-Terreni condition. However, 
 the Acquistapace-Terreni condition  requires a certain H\"older regularity of $C$ with respect to $t\in[0,T].$
  On the other hand this approach does not only cover the situation with time-dependent domains, but also
   $L^p$-maximal regularity is independent of $p\in (1,\infty)$ in this case \cite{Hi-Mo2}.
   In general, it is not clear whether  $L^q$-maximal regularity of  a family of operator $\{C(t),\ t\in (0,T)\}$ for some $q\in(1,\infty)$
  implies $L^p$-maximal regularity of $\{C(t),\ t\in (0,T)\}$ for all $p\in(1,\infty).$
   Concerning the case where the operators $C(t),\ t\in (0,T),$ have a common domain  $D,$ Pr\"uss and Schnaubelt
   \cite{Pr-Sc} and Amann \cite{Am} proved $L^p$-maximal regularity of $\{C(t),\ t\in (0,T)\}$ under the conditions that $t\mapsto C(t)$ is continuous and
    that each $C(t)$ has $L^p$-maximal regularity. This result has been generalised by  Arendt et al. \cite{ACFP} to   relative
    continuous functions $t\mapsto C(t).$
\par\noindent  Using the results of \cite{ACFP} and following an idea given in \cite{Sch-We} we prove $L^p$-maximal regularity results for
(\ref{Multi-pert1-form}) with initial data $x_0\not=0$ without assuming the Acquistapace-Terreni condition.
\par\noindent\vskip0.5cm Section \ref{MR-Hilbert spaces} is devoted to the case
where the operators  $A(t), \ t\in[0,T],$ arise from  sesquilinear forms $\mathfrak a(t,.,.)$ on a Hilbert space $H$ with a common form domain
$V$ and $B(t), \ t\in[0,T]$ are bounded linear operators on $H.$ {\it
Form methods} or {\it variational methods} give access to results of existence and uniqueness, and regularity results of the solution  in the case of variable domains and
provide  the simplest and  most  efficient way  to study parabolic evolution equations with time-dependent operators on Hilbert spaces. They were developed by T. Kato \cite{Ka} and in different but equivalent language by  J. L. Lions \cite{Lio61}. Recently a generalisation of the classical approach of Kato and Lions has been given
 by W. Arendt and T. ter Elst \cite{Ar-El}. Their approach covers in particular Dirichlet-to-Neumann operators and degenerate equations. In this present work we are concerned with the classical approach by Lions.
\par\noindent
For the case where $B\equiv I$ and $p=2,$ Lions proved $L^2$-maximal regularity of (\ref{Multi-pert1-form}) if $\fra$ is symmetric i.e., $\mathfrak a(t,u,v) = \overline{\mathfrak a(t,v,u)}$ and  $x_0=0$ (respectively  $x_0\in D(A(0))$) provided  $\mathfrak a(\cdot,u,v) \in C^1 [0,T]$ (respectively $\mathfrak a(\cdot,u,v) \in C^2 [0,T]$ and $f\in H^1(0,T;H)$) for all $u,v \in V,$  \cite[p.~65 and p.~94]{Lio61}. Bardos \cite{Bar71} also proved $L^2$-maximal regularity for $x_0\in V$ under the
assumptions that the domains of both $A(t)^{1/2}$ and $A(t)^{*1/2}$ coincide with $V$ and that $\mathcal A(\cdot)^{1/2}$ is continuously
differentiable with values in $\mathcal L(V,V'),$ where $\A(t)\in \mathcal L(V,V')$ is the operator associated with $\fra(t,.,.)$ on $V'.$
For $p\in(1,\infty)$ and $B\equiv I,$  let us mention  a result of Ouhabaz and Spina \cite{OS10} and Ouhabaz  and Haak \cite{OH14}. They proved $L^p$-maximal regularity for forms such that
$\fra(.,u,v) \in C^\alpha[0,T]$ for all $u,v \in V$  and some
$\alpha > \frac 1 2$. The result in  \cite{OS10} concerns the case $u_0=0$ and the one in \cite{OH14} concerns the case $u_0$ in the real-interpolation space
$(H,D(A(0)))_{1/p^*,p}.$
\par\noindent Left multiplicative perturbation by $B$ was recently investigated by Arendt et al. in \cite{ADLO13}.  They proved $L^2$-maximal regularity for
\begin{equation}\label{pert-Arendt}
  \dot u(t)+B(t)A(t)u(t)=f(t) \ \hbox{a.e.  on } (0,T), \ \ u(0)=u_0\in V
\end{equation}
 assuming that the sesquilinear form $\mathfrak a$ can be written as \[\mathfrak a(t,u,v) = \mathfrak a_1(t,u,v) + \mathfrak a_2(t,u,v)\]
where $\fra_1$ is symmetric, continuous,  $H$-elliptic and piecewise Lipschitz-continuous on $[0,T],$
whereas $\mathfrak a_2\colon [0,T] \times V \times H \to \mathbb C$ satisfies $|\mathfrak a_2(t,u,v)|\le M_2\|u\|_V \|v\|$
and $\mathfrak a_2(\cdot,u,v)$ is measurable for all $u\in V$, $v \in H$.
Furthermore, they assume that $B\colon [0,T] \to \mathcal L(H)$ is strongly measurable such that $\|B (t)\|_{\mathcal L(H)} \le \beta_1$ for a.e. $t \in [0,T]$
 and
$0 < \beta_0 \le (B(t)g \mid g)_H$ for $g \in H$, $\|g\|_H=1$, $t \in [0,T].$
\par\noindent In order to prove $L^2$-maximal regularity for the right multiplicative perturbation problem (\ref{Multi-pert1-form}), we need more regularity on $B.$
In addition to the conditions considered in \cite{ADLO13}, listed above, we assume that $ B:[0,T]\to \mathcal L(H)$  is piecewise Lipschitz continuous
 and selfadjoint (i.e.,  $B(t)^*=B(t)$ for all $t\in[0,T]$). Then as in Section \ref{MR-Banach spaces} we deduce  $L^2$-maximal regularity
of (\ref{Multi-pert1-form}) from the one of (\ref{pert-Arendt}).
\par\noindent  Applications to the parabolic evolution equation (\ref{dissipative system}) are presented in Section \ref{Application}.
%
\section{Perturbation  of maximal regularity in  Banach spaces}\label{MR-Banach spaces}
\subsection{Definition and preliminary}
Let  $(D,\Vert .\Vert_D)$ and $(X,\Vert .\Vert)$
be two Banach spaces such that $D\underset{d}{\hookrightarrow}X,$ i.e., $D$ is continuously and densely
embedded into $X.$  Let $A\in\mathcal{L}(D,X),$ $p\in(1,\infty)$ and $T>0$ be fixed.
We say that $A$ has  {\it $L^p$-maximal regularity} if
for every $f\in L^p(0,T; X)$  there exists
a unique function $u$ belonging to the \textit{maximal regularity
space} \[{\MR}(p,X):={\MR}(0,T,p,X)= L^p(0,T; D)\cap W^{1,p}(0,T;X)\] such that
\begin{equation}\label{CPA-0}
   \dot{u}(t)+Au(t)=f(t)\ \  \hbox{a.e. on}\ [0,T],\ \qquad u(0)=0.
\end{equation}

\noindent Recall that $W^{1,p}(0,T;X)\subset
C([0,T];X),$ so that $u(0)=0$ in (\ref{CPA-0}) is well defined. The space $\MR(p,X)$  is a Banach space with the norm
\[\|u\|_{\MR}: =\|u\|_{L^p(0,T;D)}+\|u\|_{ W^{1,p}(0,T ; X)}.\]
\par\noindent $L^p$-maximal regularity for autonomous evolution equations is a well understood property and has been intensively investigated in the literature. In the autonomous case  $L^p$-maximal regularity is independent of the bounded interval $[0,T]$ and of $p\in(1,\infty)$ \cite{Ku-We, Ca-Ve, Sob}. Thus we denote by $\mathcal{MR}$ the set of all  operators
  $A\in \mathcal{L}(D,X)$ having $L^p$-maximal regularity.
  It is well known that if $A$ has $L^p$-maximal regularity then $A$ is closed as unbounded operator on $X$ and
$-A$  generates a holomorphic $C_0$-semigroup $(T(t))_{t \geq 0}$
 on $X$ \cite{A-B 3, Do, Ku-We}. Moreover, in Hilbert spaces an operator $A$ has $L^p$-maximal regularity if and only if $-A$ generates a holomorphic $C_0$-semigroup \cite{DSi}. This equivalence is restricted to  Hilbert spaces \cite{Ka-La}.
The reader may
consult   \cite{Are04, Ku-We} for a survey and further references. \par\noindent Consider the initial value problem
\begin{equation}\label{CPA-x}
   \dot{u}(t)+Au(t)=0\ \  \hbox{a.e. on}\ [0,T],\ \qquad  u(0)=u_0.
\end{equation}
\noindent Assume that $A\in \mathcal{MR}.$ Then (\ref{CPA-x}) has the  unique solution $u(t)=T(t)u_0\in \MR(p,X)$
if and only if  $u_0$ lies in the \textit{trace space}  \[ Tr=\{u(0):\ u\in \MR(p,X)\}\]
 (see \cite{Are04}, \cite{ACFP}). The space $Tr$ is a Banach space with the norm
\[\|x\|_{Tr}:=\inf \left\lbrace \|u\|_{MR}: u(0)=x\right\rbrace.\]
Note that the trace space does neither depend on the interval $[0,T]$
 nor on the choice of the point where the
functions $u\in MR(p,X)$ are evaluated. We also recall that $Tr$ is isomorphic to the
real interpolation space
 $(X,D)_{\frac{1}{p*},p},$ where $\frac{1}{p*}+\frac{1}{p}=1$ and

  \[\MR(p,X)\underset{d}{\hookrightarrow}C([0,T];Tr).\]
%
\par\noindent  Suppose now that the operator $A$ is time-dependent and consider the non-autonomous Cauchy problem associated with $A$
\begin{equation}\label{CPNA-x}
   \dot{u}(t)+A(t)u(t)=f(t)\ \  t\hbox{-a.e. on}\ [0,T],\  u(0)=0.
\end{equation}
The $L^p$-maximal regularity for $(\ref{CPNA-x})$ is defined as follows.
\begin{definition}\label{def-max-reg-nonautonomous}  We say that (\ref{CPNA-x}) has {\it $L^p$-maximal regularity} on the bounded interval $(0,T)$ (and  write $\{A(t),\ t\in (0,T)\}\in \mathcal{MR}(p,X)$) if for each $f\in L^p(0,T;X)$ there exists a unique function $u\in W^{1,p}(0, T; X)$ such that $u(t)\in D(A(t))$ for almost every $t\in(0,T)$ and $t\mapsto A(t)u(t)\in L^p(0,T;X)$  satisfying (\ref{CPNA-x}).
\end{definition}
Assume that $D(A(t))=D$ for almost every $t\in [0,T]$ and  $A:[0,T]\longrightarrow \L(D,X)$ is strongly measurable.
 Recall that the function $A:[0,T]\longrightarrow \L(D,X)$ is {\it relatively
  continuous} (\cite[Definition 2.5]{ACFP}) if for each $t\in[0,T]$ and  all
  $\varepsilon>0$ there exist $\delta> 0,\ \eta\geq
  0$ such that for all $s\in[0,T],\ \vert t-s\vert \leq \delta$
implies that
  \[
  \Vert A(t)x-A(s)x\Vert\leq \varepsilon \Vert x \Vert_D+
  \eta\Vert x\Vert\ \hbox{ for } x\in D.\]
 If A is relatively continuous then $A$ is
 bounded (see \cite[Remark 2.6]{ACFP}).

\par\noindent The following lemma is used in
the next sections and is easy to proof.
\begin{lemma}\label{rel-cont-pert}Let $A: [0,T]\rightarrow \mathcal{L}(D,X)$ be relatively continuous and  $B: [0,T]\rightarrow \mathcal{L}(X)$ be another function. Then the following holds.

\begin{enumerate}
 \item[{\rm a)}] If $B$ is norm continuous, then  the product $BA$ is also relatively continuous.
 \item[{\rm b)}] If  $B$ ist bounded, then $A+B$ is  relatively continuous 	
\end{enumerate}
\end{lemma}
%

\subsection{Perturbation of $L^p$-maximal regularity }\label{Section-Perturbation I}
 \par Let $X,D$ be the Banach spaces as in the previous section and additionally we assume  that $X$ has the Radon-Nikodým property
.\par\noindent  Let $A :[0,T]\rightarrow \mathcal{L}(D,X)$ be strongly measurable and relatively continuous.
In this section we prove some perturbation results for the problem (\ref{CPNA-x}). Let
$B :[0,T]\rightarrow\mathcal{L}(X)$  such that $B(.)x\in C^1([0,T],X)$  for each  $x\in X,$  the inverse $B(t)^{-1}\in \mathcal{L}(X)$ exists for every $t\in [0,T]$ and   $\sup_{t\in[0,T]}\|B(t)^{-1}\|_{\mathcal{L}(X)}<\infty.$ Consider the following  non-autonomous problem

\begin{equation}\label{Mult-pert-Pbm}
 \dot{u}(t)+A(t)B(t)u(t)=f(t)\ \  \hbox{a.e. on}\ [0,T],\ \qquad  B(0)u(0)=x.
\end{equation}
\newline Here the operators $A(t)B(t)$ are defined on their natural domains, namely \[\mathcal{D}_t:={D}(A(t)B(t))=\{x\in X:\ B(t)x\in D(A(t))\}\]
In contrast to $D(A(t))$ the domains $\mathcal{D}_t$ generally  depend on the time variable. The general question is whether the problem (\ref{Mult-pert-Pbm}) is well posed in $L^p$ with maximal regularity.\par\noindent   By  $\mathfrak{AB}$ we denote the multiplication operator  on $L^p(0,T; X)$ defined by
\begin{equation*}
\begin{aligned}
 (\mathfrak{AB}u)(t)&=A(t)B(t)u(t)\ \  \hbox{ for almost every } t\in[0,T].
\\ D(\mathfrak{AB})&=\{u\in L^p(0,T;X):\ u(t)\in \mathcal{D}_t \hbox{ a.e and }  \mathfrak{AB}u\in L^p(0,T;X)\}
\end{aligned}
 \end{equation*}
  \par\noindent Note that if $A(t)$ is closed for almost every $t\in  [0,T]$ then $(\mathfrak{AB}, D(\mathfrak{AB}))$ is closed. In this case the maximal regularity space $\MR_B(p,X)$ given by
 \[\MR_B(p, X):=\MR_B(0, T, p, X):=D(\mathfrak{AB})\cap W^{1,p}(0, T; X)\]
 is a Banach space with the norm
 \[\|u\|_{MR_B}:=\|\dot u\|_{L^p(0,T;X)}+\|u\|_{L^p(0,T;X)}+\|\mathfrak{AB}u\|_{L^p(0,T;X)}.\]
 For each interval $[a,b]\subset [0,T]$,  we may consider the operator $\mathfrak {AB}$ on $L^p(a,b;X)$ In order to keep notation simple, we do not use different notations here.
 \begin{remark}\label{Radon-Nykodim}
Since  the Banach space $X$ has the Radon-Nikodým, the
space of  absolutely continuous functions on $[0,T]$ with
values in $X$ is the same as the Sobolev space $W^{1,1}(0,T;X)$
and $\frac{d}{dt}u:=\dot u$ coincides with the weak derivative. The function $u$ is in $W^{1,p}(0,T;X)$ if and only if $u
\in W^{1,1}(0,T;X)$ and $\dot u \in L^p(0,T;X)$ (see e.g., \cite[Section 1.2]{ABHN}).
\end{remark}
\par For the following lemma see \cite[Lemma 4.3]{ADLO13}.
\begin{lemma}\label{lemma:lipschitz_continuous_operators}
    Let $B:[0,T]\to \L(X)$ be Lipschitz continuous.
    Then the following holds.
    \begin{enumerate}[label={\rm \alph*)}]
        \item There exists a bounded, strongly measurable function
            $\dot B : [0,T] \to \L(X)$ such that
            \[
                \frac \d {\d t} B(t)x = \dot B(t)x \quad (x \in X)
            \]
            for a.e.\ $t\in [0,T]$ and
            \[
                \norm{\dot B(t)}_{\L(X)} \le L \quad (t \in [0,T])
            \]
            where $L$ is the Lipschitz constant of $B$.
        \item If $ u\in W^{1,p}(0,T;X)$  then
            $Bu := B(.)u(.) \in W^{1,p}(0,T;X)$ and
            \begin{equation*}\label{eq:chain_rule}
                (Bu)\dot{} = \dot B(.)u(.)+ B(.) \dot u(.).
              \end{equation*}
         \item If $ u\in W^{1,p}(0,T;X)$, then
            $B^{-1}u := B^{-1}(.)u(.) \in W^{1,p}(0,T;X)$ and
            \begin{equation*}\label{eq:chain_rule}
                (B^{-1}u)\dot{} = B^{-1}\dot B(.)B^{-1}(.)u(.)+ B^{-1}(.) \dot u(.).
            \end{equation*}
    \end{enumerate}
\end{lemma}
\vskip0.5cm\par  Note that if $A(t)$ is closed for almost every $t\in  [0,T]$ then $\MR(p,X)=B(\MR_B(p,X))$ and for all $u\in \MR_B(p,X)$ and $v\in MR(p,X)$ we have
 \begin{equation}\label{estimation}\|Bu\|_{MR}\leq c_1 \|u\|_{MR_B}\qquad\text{ and } \qquad\|B^{-1}v\|_{MR_B}\leq c_2 \|v\|_{MR} \end{equation}
where  $c_1:=\sup\{\|B\|+L,1)$ and $c_2:=\sup\{\|B^{-1}\|+\|B^{-1}\|^2L,1).$ In particular, if $B=I$ then $MR(p,X)=MR_I(p,X)$ coincide.
\begin{proposition} \label{Proposition-sub-interval-MR} Assume that $\{A(t)B(t), t\in[0,T]\}$ has  $L^p$-maximal regularity on $(0,T')$ for every $T'\in(0,T].$ Then for every $s\in[0,T)$ and every $(f,x_0)\in L^p(s,T;X)\times Tr$ there exists a unique $u\in \MR_B(s, T, p, X)$ such that \begin{equation}\label{Mult-pert-Pbm-Initial-value}
 \dot{u}(t)+A(t)B(t)u(t)=f(t)\ \  \hbox{a.e. on}\ [s,T],\ \qquad   B(s)u(s)=x_0.
\end{equation}
\end{proposition}
\begin{proof}
Let $(f,x_0)\in L^p(s,T;X)\times Tr.$ Let $\omega\in \MR(0, T, p, X)$ such that $\omega(0)=x_0.$ Let $\omega_s(t):=\omega(t-s)$ for $t\in[s,T].$ Thus $\bar\omega_s:=B^{-1}\omega_s\in MR_B(s, T, p, X).$
Let $\bar f_s\in L^p(0,T;X)$ defined on $[0,s)$ by $f_s=0$ and by 
\[f_s:=-\dot{\bar{\omega}}_s(.)-A(.)B(.)\bar\omega_s(.) +f(.)
      \hbox{ on } \  [s,T].\]
Denote by $v_s\in MR_B(s, T ,p , X)$  the unique solution of the problem
\begin{equation*}
 \dot{v}_s(t)+A(t)B(t)v_s(t)=\bar f_s(t)\ \  \hbox{a.e. on}\ [0,T],\ \qquad  v_s(0)=0.
\end{equation*}
By $L^p$-maximal regularity and the fact that $\bar f_s=0$ on $(0,s),$ $v_s=0$ on $[0,s].$ Set $u_s(t):=v_s(t)+\bar{\omega}_s(t)$ for $t\in[s,T].$ Then $u_s\in MR_B(s ,T ,p , X)$ and solves (\ref{Mult-pert-Pbm-Initial-value}).
\par\noindent Let $u_1,u_2\in MR_B(s, T, p, X)$ be two solution of (\ref{Mult-pert-Pbm-Initial-value}). Then
\[\bar v
(t):=\left\{%
\begin{array}{ll}
    0 & \hbox{ if } \  0\leq t<s,\\
    u_1-u_2 & \hbox{ if } \  s\leq t\leq T, \\
\end{array}%
\right. \]
is a solution of (\ref{Mult-pert-Pbm})
on $(0,T)$ for inhomogeneity  $f=0$ and $x_0=0.$ Thus by  maximal regularity $v=0.$
\end{proof}
In the following theorem we give a sufficient conditions for  $L^p$-maximal regularity of (\ref{Mult-pert-Pbm}).
\begin{theorem}\label{per-thm-Banach}
Assume that  $B(t)A(t)\in \mathcal{MR}$ for every $t\in [0,T].$
Then the problem $(\ref{Mult-pert-Pbm})$ has $L^p$-maximal regularity on $(0,T')$ for all $T'\in[0,T)$ and  $p\in (1,\infty).$ In particular, for each $(f,x_0)\in L^p(0,T;X)\times Tr$ there exists a unique $u\in \MR_B(p, X)$ satisfying
\begin{equation}\label{per-thm-Banach-3}
 \dot{u}(t)+A(t)B(t)u(t)=f(t)\ \  \hbox{a.e. on}\ [0,T],\ \qquad  B(0)u(0)=x_0.
\end{equation}
Moreover, $B(.)u(.)\in C([0,T];Tr).$
\end{theorem}

\begin{proof} For every fixed $t\in[0,T]$ we apply Proposition 1.3 in \cite{ACFP} to $\tilde{A}=B(t)A(t)$ and  $\tilde{B}(.)=-\dot{B}(t)B(t)^{-1}$ and obtain  $B(t)A(t)-\dot{B}(t)B(t)^{-1}\in \mathcal{MR}$ for every   $t\in[0,T].$  Moreover,  $B(.)A(.)-\dot{B}(.)B(.)^{-1}$ is strongly measurable and relatively continuous by Lemma \ref{rel-cont-pert}.
Thus  Theorem 2.7 in \cite{ACFP} implies that  $B(.)A(.)-\dot{B}(.)B(.)^{-1}\in\mathcal{MR}(p,X).$
 Let now $f\in L^p(0,T;X).$ Let $v\in {MR}(p,X)$ be the unique solution of
 \begin{equation}\label{prb2-proof}\dot{v}+B(t)A(t)v-\dot{B}(t)B(t)^{-1}v=Bf\ \ \ \hbox{a.e. on }\ [0,T],
\ \qquad  v(0)=0\end{equation}
and set $u(t):=B(t)^{-1}v(t)$ for $t\in[0,T].$  Observe that $u(t)\in {D}_t$ a.e. and $ A(\cdot)B(\cdot)u(\cdot)\in L^p(0,T;X)$ since
$v(t)\in D$ and $A(t)B(t)u(t)=A(t)v(t)$   a.e.
From $$\frac{d}{dt}B(t)^{-1}x=-B(t)^{-1}\dot{B}(t)B(t)^{-1}x\ \  (x\in X,\  t\in[0,T])$$ and  since $X$ has the Radon-Nikodým
property, we have that $u$ is absolutely continuous and
\begin{align*}\dot{u}(t)&=\frac{d}{dt}(B(.)^{-1}v)(t)
\\&=-B(t)^{-1}\dot{B}(t)B(t)^{-1}v(t)+B(t)^{-1}\dot{v}(t)
\\&=-B(t)^{-1}\dot{B}(t)B(t)^{-1}v(t)+B(t)^{-1}\big(B(t)f(t)-B(t)A(t)v(t)
+\dot{B}(t)B(t)^{-1}v(t)\big)
\\&=f(t)-A(t)B(t)u(t).
\end{align*}
Thus $u\in \MR_B(p; X)$ and satisfies
\begin{equation}\label{Prb1-proof}
\dot{u}(t)+A(t)B(t)u(t)=f(t)\ \  \hbox{a.e. on}\  [0,T],\ \
\ \qquad u(0)=0.
\end{equation}
\noindent The uniqueness of solvability of (\ref{Prb1-proof}) follows from the one of (\ref{prb2-proof}).
\par\noindent The last assertion follows from the fact that $Bu=v\in MR(p, X)$ and the embedding $MR(p,X)\underset{d}{\hookrightarrow}C([0,T];Tr)$.
\end{proof}
\par We consider now an intermediate Banach space $Y,$ i.e., $D\underset{d}\hookrightarrow Y\underset{d}\hookrightarrow X$ such that for each $\varepsilon>0$ there exists
$\eta\geq 0$  such that
\[ \|x\|_Y\leq \varepsilon \|x\|_D+\eta\|x\|,\ \quad x\in D,\]
We then say that  $Y$ is \textit{close to $X$ compared with $D$} see \cite{ACFP}. Then we have the  following perturbation result.
\begin{proposition}\label{per-thm-Banach2}
Let $A:[0,T]\rightarrow \mathcal{L}(D,X)$  and $B:[0,T]\rightarrow \mathcal{L}(X)$ are as in Theorem \ref{per-thm-Banach}. Let $C:[0,T]\rightarrow \mathcal{L}(Y,X)$ be strongly measurable and bounded.
Then for each $(f,x_0)\in L^p(0,T;X)\times Tr$ there exists a unique $u\in MR_B(p,H)$ such that
\begin{equation}
   \begin{aligned}
  \dot{u}(t)+A(t)B(t)u(t)+C(t)u(t)&=f(t)\ \  \hbox{a.e. on}\ [0,T],\ \
\\ B(0)u(0)&=x_0.
\end{aligned}
\end{equation}
Moreover, $B(.)u(.)\in C([0,T];Tr).$
\end{proposition}
\begin{proof}
The proof is the same as the proof of  Theorem \ref{per-thm-Banach}. Replacing (\ref{prb2-proof}) in the proof of Theorem \ref{per-thm-Banach} by
\begin{equation}\label{prb2-proof2}
  \begin{aligned}
\dot{v}+B(t)A(t)v-\dot{B}(t)B(t)^{-1}v+B(t)C(t)v&=B(t)f\ \ \ \hbox{a.e. on }\ [0,T],
\ \  \\v(0)&=0,
  \end{aligned}
\end{equation}
we have only to show that $B(.)A(.)-\dot{B}(.)B(.)^{-1}+B(.)C(.)\in\mathcal{MR}(p, X),$ which is true by \cite[Theorem 2.11]{ACFP}.
\end{proof}
In the last part of this section we study the existence of the evolution family associated with the non-autonomous evolution equation (\ref{Mult-pert-Pbm}). Let $\Delta:=\{ (t,s)\in[0,T]\times [0,T]: t\geq s\}.$ Recall that a family of linear operators $(U(t,s))_{(t,s)\in\Delta}$ is a \textit{strongly continuous evolution family} on a Banach space $Y\subseteqq X$ if the following properties holds.
\begin{description}
\item[$(i)$] $U(t,s)\in\L(Y)$ for every $(t,s)\in\Delta,$
\item[$(ii)$] $U(t,t)=I$ and $U(t,s)=U(t,r)U(r,s)$ for every $0\leq s\leq r\leq t\leq T,$ and
\item [$(iii)$] for every $x\in Y$ the function $U(.,.)x$ is continuous  on $\Delta$ with value in $Y.$
\end{description}
Assume that $A$ and $B$ satisfy the hypothesis of Theorem \ref{per-thm-Banach-3}.
We have seen in the proof of Theorem \ref{per-thm-Banach-3} that $u\in MR_B(s,T,p,X)$ satisfies
 \begin{equation}\label{Mult-pert-Pbm-Initial-value2}
 \dot{u}(t)+A(t)B(t)u(t)=0\ \  t\hbox{-a.e. on}\ [s,T],\ \   B(s)u(s)=x_0.
\end{equation}
if and only if $v:=B(.)u\in MR(s,T,p,X)$ satisfies
  \begin{equation}\label{Multi-pert-Left2}\dot{v}+B(t)A(t)v-\dot{B}(t)B(t)^{-1}v=0\ \ \ \hbox{a.e. on }\ [s,T],
\ \  v(s)=x_0.\end{equation}
For every $(t,s)\in\Delta$ and every $x_0\in Tr$ we can define \[U(t,s)x_0:=v(t),\]
where $v$ is the unique solution of (\ref{Multi-pert-Left2}).
By \cite[Proposition 2.3, Propossition 2.4]{ACFP}  the family $(U(t,s))_{(s,t)\in \Delta}$
is a bounded and strongly continuous evolution family on $Tr$ and for all $f\in L^p(0,T;Tr),$
\[ v(t)=\int_0^tU(t,r)f(r)dr\]
is the unique solution of the inhomogeneous problem
\begin{equation*}\dot{v}+B(t)A(t)v-\dot{B}(t)B(t)^{-1}v=f(t)\ \ \ \hbox{a.e. on }\ [0,T],
\ \qquad  v(0)=0.\end{equation*}
Then we have the following result.
\begin{corollary}\label{evolution-family}
Let $f\in  B^{-1}L^p(0,T;Tr)$ and $u_0:=B^{-1}(0)x_0\in B^{-1}(0)Tr.$ Then the unique solution $u$ of $(\ref{per-thm-Banach-3})$
is given by \[u(t)=B^{-1}(t)U(t,0)B(0)u_0+\int_0^t B^{-1}(t)U(t,r)B(r)f(r)dr.\]
\end{corollary}

Now  assume in Theorem \ref{per-thm-Banach-3} that $A$ is norm-continuous.  Then by \cite[Theorem 3.1]{Po-St} there exists  a bounded, strongly continuous evolution family $(V(t,s))_{(t,s)\in\Delta}$ on $X$ which maps $X$ into $(X,D)_{\frac{1}{p^*},p}\cong Tr.$ Moreover,  the solution $v$  of
 \begin{equation}\label{Multi-pert-Left23}\dot{v}+B(t)A(t)v-\dot{B}(t)B(t)^{-1}v=f\ \ \ \hbox{a.e. on }\ [0,T],
\ \  v(0)=x_0.\end{equation}
for  $x_0\in Tr$ and $f\in L^p(0,T;X)$ is given by
\[v(t)= V(t,0)x_0+\int_0^t V(t,r)f(r)dr,\ \ t\in [0,T].\]
Clearly the evolution family $V$ coincides with $U$ on $Tr.$ As a consequence we obtain the following.
\begin{corollary}\label{evolution-family-X}
Assume that $A$ is norm continuous. Then the family $(\Phi(t,s))_{(t,s)\in\Delta}$ given by
 \[\Phi(t,s):=B^{-1}(t)V(t,s)B(s)\]is a bounded, strongly continuous evolution family on $X.$ Moreover, for each $f\in L^p(0,T;X)$ and $u_0:=B^{-1}(0)x_0\in B^{-1}(0)Tr.$ Moreover, the unique solution $u$ of $(\ref{per-thm-Banach-3})$
is given by \[u(t)=\Phi(t,0)u_0+\int_0^t \Phi (t,r)f(r)dr.\]

 \end{corollary}

\begin{remark}
The previous  results was proved in \cite{ACFP} and \cite{Po-St} in the case where $B=I.$
\end{remark}

%
\section{Evolution equations governed by  forms}\label{MR-Hilbert spaces}
Throughout this section $H,V$ are two separable  Hilbert spaces over $\mathbb K=\mathbb C$ or $\mathbb R.$ We denote by $(\cdot \mid \cdot)_V$ the scalar product and $\|\cdot\|_V$ the norm
on $V$ and by $(\cdot \mid \cdot), \|\cdot\|$ the corresponding quantities in $H.$ Moreover, we assume that
$V \underset d \hookrightarrow
H.$ 
Let $V'$ denote  the antidual of $V$ if $\mathbb K=\mathbb C$ and the dual if $\mathbb K=\mathbb R.$  The duality
between $V'$ and $V$ is denoted by $\langle \cdot, \cdot \rangle$. As usual, by identifying $H$ with  $H',$ we have   $V\hookrightarrow H\cong H'\hookrightarrow V'$
with continuous and dense embeddings (see, e.g., \cite{Bre11}).
%
\subsection{Forms and associated operators}

Consider a  \emph{continuous} and $H$-\emph{elliptic} sesquilinear form $\fra:V \times V \to \K$. This means,
respectively
\begin{equation}\label{eq:a_continuous}
    \abs{\fra(u,v)} \le M \norm u _V \norm v _V \quad \hbox{ for some }  M\geq 0 \hbox{ and all } u,v \in V,
\end{equation}
\begin{equation}\label{eq:H-elliptic}
    \Re \fra(u) + \omega \norm u^2 \ge \alpha \norm u _V^2 \quad \hbox{ for some   } \alpha>0,~\omega\in \R\hbox{ and all } u \in V.
\end{equation}
Here and in the following we shortly write $\fra(u)$ for $\fra(u,u).$ The form $\fra$ is called {\it coercive} if $\omega=0$ and {\it symmetric} if
$\fra(u,v)=\overline{\fra(v,u)}$ for all $u,v\in V.$ By the Lax-Milgram theorem, there exists an isomorphism $\A: V\longrightarrow V'$ such that
$\langle \A u,v\rangle =a(u,v)$ for all $u, v\in V.$ It is well known that $-\A$ generates a bounded holomorphic $C_0$-semigroup on $V'.$ In the case where $\K=\R$ this means that the $\C$-linear
extention of $-\A$ on the complexification of $V'$ generates a holomorphic $C_0$-semigroup. We call $\A$ {\it the operator associated with $\fra$ on $V'$}. In applications to  boundary valued problems, the operator $\A$ does not realize the boundary conditions in question. For the latter, we have to consider {\it the operator $A$ associated with $\fra$ on $H$}:
\begin{align*}
D(A):=&\{ u\in V: \exists f\in H \hbox{ such that }  \mathfrak a(u,\psi)=(f\mid \psi) \hbox{ for all  } \psi\in V\}
\\ Au:=&f.
\end{align*}
Note that $v$ is uniquely determined by $u$ since $V$ is dense in $H.$ Moreover, it is easy to see that $A$ is the part of $\A$ in $H,$ i.e.,\
\begin{align*}
    D(A) := {}& \{ u\in V : \A u \in H \}\\
    A u = {}& \A u.
\end{align*}
\begin{theorem}\label{characterization of op induced by form}
Let $A$ be an operator on $H.$ Then the following are equivalent.
\begin{itemize}
 \item[(i)] $A$ is associated with a continuous and $H$-ellipitic form $\fra:V \times V \to \K.$
\item[(ii)] There exist  $\omega\in\R$  and $\theta\in(0,\frac{\pi}{2})$ such that:

   (a) $(\omega+A)D(A)=H$,

   (b) $e^{\pm i\theta }(\omega+A)$ are  accretive.
\item[(iii)] $-A$ generates a holomorphic $C_0$-semigroup $T$ of angle $\theta \in(0,\frac{\pi}{2})$ such that  for some $\omega\in \R$
\[ \|T(z)\|_{\mathcal L(H)}\leq e^{\omega \mid z\mid},\quad z\in\Sigma_\theta:=\{re^{i\alpha} : r>0, |\alpha|<\theta\}. \]
\end{itemize}
\end{theorem}
\noindent
 \par For all results above we refer to, e.g.\ \cite[Chapter 2]{Tan79}, \cite[Section 5]{Are04} and \cite[Chapter 1]{Ouh05}. The definition of the operator $A$ on $H$ associated with $\fra$ depends on the scalar product considered on $H,$ i.e.,  equivalent scalar products leads to different operators. \begin{proposition}\label{Mult-pert-Auto}
Let $\fra$ be a  continuous and $H$-elliptic form on $V$ and let $A$ be the associated operator on $H$. Let  $B\in \mathcal{L}(H)$ be self-adjoint such that
\begin{equation}\label{strictly positive}(Bx\mid x)\geq \beta \lVert x\lVert^2\ \ \ (x\in H)\end{equation}
for some  $\beta>0.$ Then $-AB$ generates a holomorphic $C_0$-semigroup on $H.$
\end{proposition}
\begin{proof} Let $H_B$ be the Hilbert space $H$ endowed with the
 scalar product \[(u\mid v)_B:=(B^{-1}u\mid v).\] By (\ref{strictly positive}) this scalar product induces an equivalent norm on $H.$ It is easy to see that  $BA$ is the operator associated with $\fra$ on $H_B$ \cite[Section 5.3.5]{Are04}. Then $-BA$  and, by similarity, $-AB$ generates a holomorphic $C_0$-semigroup on $H.$
\end{proof}

%
\subsection{Perturbation of non-autonomous maximal regularity in Hilbert spaces }
In this section we extend Proposition \ref{Mult-pert-Auto} to the non-autonomous setting.
Let $T>0$ and
\[
    \fra :[0,T]\times V \times V \to \mathbb K\quad\text{ and } \quad B: [0,T]\rightarrow \mathcal{L}(H).
\]
Throughout this section will make the followings assumptions on $\fra$ and $B$. As in \cite{ADLO13} we assume  that $\fra$ can be written as the sum of two non-autonomous forms
\begin{equation}\label{decomposition}\fra(t,u,v)=\fra_1(t,u,v)+\fra_2(t,u,v)\quad (t\in [0,T], u,v\in V)\end{equation}
where $\fra_1(t,u,v): [0,T]\times V \times V \to \mathbb K$ is such that

\begin{equation}\label{continuity-form1}
	|\mathfrak a_1(t,u,v) |\le M_1 \|u\|_V \|v\|_V \quad (t\in[0,T], u,v\in V)
\end{equation}
for some $M_1 \ge 0,$ and
 \begin{equation}\label{coercivite-form1}
{\rm Re}~ \mathfrak a_1 (t,u,u) +\omega\|u\|^2\ge \alpha \|u\|^2_V \quad (t\in [0,T], u\in V)
\end{equation}
for some $\alpha > 0$ and $\omega\in \mathbb R.$  We also assume that $\fra_1$ is symmetric, i.e.,
\begin{equation}\label{symmetrie}\fra_1(t,u,v)=\overline{\fra_1(t,v,u)} \qquad (t\in[0,T], u,v\in V).
\end{equation}
Further we suppose that $\fra_1$ is \emph{Lipschitz continuous} in $t\in [0,T]$, i.e., there exists $L_1>0$ such that
\begin{equation}\label{lipschitz-continuous-form}
    \abs{\fra_1(t,u,v)- \fra_1(s,u,v)} \le L_1 \abs{t-s} \norm u_V \norm v_V \quad (t,s \in [0,T],\  u,v \in V),
\end{equation}
whereas $\mathfrak a_2\colon [0,T] \times V \times H \to \mathbb K$ satisfies
\begin{equation}
|\mathfrak a_2(t,u,v)|\le M_2\|u\|_V \|v\|\quad (t\in[0,T], u\in V, v\in H)
\end{equation}
for some $M_2 > 0$ and $\mathfrak a_2(\cdot,u,v)$ is measurable for all $u\in V$, $v \in H.$ We denote by $A(t)$ the operator associated with  $a(t,u,v)$  on $H$.
\
\par\noindent Let $B: [0,T]\rightarrow \mathcal{L}(H)$ be a  Lipschitz continuous function with Lipschitz constant  $L_2>0.$
Assume that $B$ is self-adjoint and uniformly  positive, i.e., $B(t)^*=B(t)$ and \[(B(t)x\mid x)\geq \beta \lVert x\lVert_H^2\] for some  constant $\beta>0$ and for all $t\in[0,T]$    and $x\in H.$
The main result of this section reads as follows.
\begin{theorem}\label{Multiplicative-perturbation-nonautonomous-Form}  The family $\{A(t)B(t),\ t\in (0,T)\}\in \mathcal{MR}(2,H).$  Moreover, for all $x_0\in V$ there exists a unique $u\in MR_B(2,H)$ with
    \begin{align}\label{Non-aut-pbl-Form1}
         \dot u(t) + A(t)B(t)u(t)  &= f(t), \quad \text{a.e.} \ t\in[0,T]\\\label{Non-aut-pbl-Form2}
                         B(0)u(0)    &=x_0.
    \end{align}
 Moreover,  $ B(\cdot)u(\cdot) \in C([0,T];V)$ and
  \begin{equation}
   	\|u\|_{ MR_B} \leq c \Big[ \|u_0\|_V + \| f\|_{L^2(0,T;H)} \Big]
    \end{equation}
    where the constant  $c=c(L_1, c_H, T, L, M, \alpha, \beta)$ is independent of $x_0$ and $f.$
\end{theorem}
\begin{proof} Let $x_0 \in V$, $f \in L^2(0,T; H)$.
By assumption on $B$ and Lemma \ref{lemma:lipschitz_continuous_operators}, $B^{-1}\dot BB^{-1}:[0,T]\rightarrow \mathcal{L}(H)$ is  bounded and for each $u\in H$ we have that
$t\mapsto  B(t)^{-1}\dot {B}(t)B(t)^{-1}u$ is weakly  measurable. Then applying \cite[Theorem 4.2]{ADLO13} to
 \[\tilde\fra:=\fra-(B^{-1}\dot BB^{-1}.\mid.)\]
we deduce that $\{B(t)A(t)-\dot {B}(t)B(t)^{-1},\ t\in (0,T)\}\in \mathcal{MR}(2,H)$ and  the non-autonomous Cauchy problem
\begin{eqnarray}\label{Left-Mult-pert-Pbm}
 \dot{v}(t)+(B(t)A(t)-\dot{ B}(t)B^{-1}(t))v(t)&=&B(t)f(t)\ \ \
 \\\label{Left-Mult-pert-Pbm}  v(0)&=&x_0
\end{eqnarray}
has a unique solution $v\in H^1(0,T;H)\cap L^2(0,T;V)$ such that $B(\cdot)u(\cdot)$ belongs to $C([0,T];V).$ Using Lemma \ref{lemma:lipschitz_continuous_operators}, the second part of the proof is the same as in the proof of the Theorem \ref{per-thm-Banach}.
\par\noindent The last assertion follows from  (\ref{estimation})  and estimate (4.1) in \cite[Theorem 4.2]{ADLO13}.
\end{proof}
We say that  $B:[0,T] \to \L(H)$ is \emph{piecewise Lipschitz-continuous} if
there exist $0= t_0 < t_1 < \dots < t_n =T$ such that on each interval $(t_{i-1}, t_i)$  the restriction of $B$ is Lipschitz-continuous on $(t_{i-1}, t_i ).$ Then the following corollary follows easily from  Theorem \ref{Multiplicative-perturbation-nonautonomous-Form}.

\begin{corollary}\label{Cor-Multiplicative-perturbation-nonautonomous-Form}
    Assume instead of the Lipschitz continuity  that $B:[0,T] \to \L(H)$ is merely piecewise Lipschitz-continuous. Then the family $\{A(t)B(t),\ t\in (0,T)\}\in \mathcal{MR}(2,H).$  Moreover, for all $x_0\in V$ the exists a unique $u\in MR_B(2,H)$ satisfies
    \begin{align*}
         \dot u(t) + A(t)B(t)u(t)  &= f(t) \quad \text{a.e.}\ t\in [0,T]\\
                         B(0)u(0)    &=x_0.
    \end{align*}
    Moreover,  $B(\cdot)u(\cdot) \in C([0,T];V).$
\end{corollary}
Theorem \ref{Multiplicative-perturbation-nonautonomous-Form} and Corollary \ref{Cor-Multiplicative-perturbation-nonautonomous-Form} are restricted to the case $p=2.$ For the general case $(p\in(1,\infty))$ we give a result under the additional assumption that the domain $D(A(t))=D$ of the operators induced by the forms $\fra(t,.,.)$ are $t$-independent. However, the  domains of the perturbed operator $A(t)B(t)$ \[D(A(t)B(t)):=\{ x\in H: B(t)x\in D\}\] may depend on the  time variable $t.$  For this we use the results of Section \ref{MR-Banach spaces}. In fact, the following results is a consequence of Theorem \ref{per-thm-Banach} and Proposition \ref{Mult-pert-Auto}.

\begin{theorem}\label{Lp-max-reg-Hilbert} Assume that  $A:[0,T]\rightarrow \mathcal{L}(D,H)$ is relatively continuous and $B:[0,T]\longrightarrow\L(H)$ is piecewise Lipschitz-continuous. Then for every $(f,x_0)\in L^p(0,T;H)\times Tr$ there exists a unique $u\in MR_B(p,H)$ such that

\begin{equation}
   \begin{aligned}
  \dot{u}(t)+A(t)B(t)u(t)&=f(t)\ \  t\hbox{-a.e. on}\ [0,T],\ \
\\ B(0)u(0)&=x_0.
\end{aligned}
\end{equation}
Moreover, $B(.)u(.)\in C([0,T];Tr).$
\end{theorem}

\begin{remark} Theorem \ref{Lp-max-reg-Hilbert} its corollary and Theorem \ref{Multiplicative-perturbation-nonautonomous-Form} remain true if we assume piecewise Lipschitz continuity  of $\fra_1$.
\end{remark}
%
\section{A general class of parabolic equations}\label{Application}
This section is devoted to an application of our results on $L^p$-maximal regularity to the non-autonomous partial differential equation (\ref{dissipative system}) with time dependent coefficients.
\par\noindent
\subsection{Description and assumptions}\label{Example: general class}
Let $1\leq k\leq n$ and $0\leq r\leq 2k.$ For $r=0$ we will use the  notations $\K^0:=\{0\}, \K^{0\times 2k}:=\L(\K^{2k},\{0\})$ and $\K^{2k\times 0}:=\L(\{0\},\K^{2k}).$ Let $T>0.$ As example we consider   the linear parabolic system
\begin{align}\label{dissipative system2}
  \partial_t u(t,\zeta)+\A(t,\zeta,\partial)\H(t,\zeta)u(t,\zeta)&=f(t,\zeta),\qquad \qquad \zeta\in[0,1],\  0\leq t\leq T,
\\ \H(0,\zeta)u(0,\zeta)&=x_0(\zeta),\qquad \qquad \ \zeta\in [0,1],
\\ \label{dissipative system2-BC1}F^*\B_{\partial}(t)(\H(t) u)&=-W_R(t)F^*\mathcal C_{\partial}(\H(t) u),\ \qquad  0\leq t\leq T,
\\ \label{dissipative system2-BC2}(I-FF^*)\mathcal C_{\partial}(\H(t) u)&=0, \ \ \qquad\qquad\qquad\qquad\qquad  \  0\leq t\leq T
  \end{align}
in $H:=L^2(0, 1; \mathbb K^n),$ where
\[\A(t,\zeta,\partial):=-\frac{\partial}{\partial \zeta}\big(GS(t)\frac{\partial}{\partial \zeta}G^*+P_1\big)-P_0,\]

\begin{equation*}\label{Boundary-condition-no1}
\B_{\partial}(t)(\H(t) u):=\left[
           \begin{array}{c}
            G^*\big(GS(t)\frac{\partial }{\partial\zeta}G^*\H(t,1)u(t,1)+P_1\H(t,1)u(t,1)\big) \\
             -G^*\big(GS(t)\frac{\partial }{\partial\zeta}G^*\H(t,0)u(t,0)+P_1\H(t,0)u(t,0)\big)\\
           \end{array}\right]
\end{equation*}
and
\begin{equation*}\mathcal C_{\partial}(t)(\H(t) u):=\left[
                                        \begin{array}{c}
                                       G^*\H(t,1)u(t,1) \\
                                       G^*\H(t,0)u(t,0) \\
                                       \end{array}
                                            \right]
\end{equation*}
We always assume the following.
 \begin{assumption}\label{Assumption}\end{assumption}

\begin{enumerate}
\item $G\in \K^{n\times k}$  has full rank and $GG^*\in \K^{n\times n}$ is a projection.
\item $P_0\in L^{\infty}(0,1; \mathbb K^{n\times n}).$
\item $P_1\in W^{1,\infty}(0,1;\K^{n\times n})$ and for some $\kappa>0$
 \begin{equation}\label{couplung-cond}
 |(I-GG^*)P_1(\zeta)u|\leq \kappa_1 |GG^*u|  \ \text{ for all }u\in \K^n, \ a.e. \text{ on } (0,1)
 \end{equation}
\item $\H:[0,T]\times [0,1]\longrightarrow \K^{n\times n}$ is self-adjoint, uniformly positive, i.e., $\H(t,\zeta)^*=\H(t,\zeta)$ and
    $0<m_1I\leq \mathcal H(t,\zeta)\leq M_1I\  (t\in [0,T], \zeta\  a.e.\in [0,1])$ for some constants $m_1,M_1>0$  and Lipschitz continuous w.r.t. the first variable such that
\[\vert \mathcal H(t,\zeta)-\mathcal H(s,\zeta)\vert\leq L_1\vert t-s\vert \ \quad (t\in [0,T], a.e. \zeta\in [0,1])\]
 for some constant $L_1>0.$
\item $S:[0,T]\times [0,1]\longrightarrow \K^{k\times k}$  satisfies properties analogous to those of $\H$ with corresponding constants $m_2,M_2$ and $L_2.$
\item $F\in \K^{2k\times r}$ has full rank and $FF^*\in\K^{2k\times 2k}$ is a projection.
\item $W_R(t):[0,T]\longrightarrow \K^{r\times r}$ is   Lipschitz continuous with $W_R(t)=W_R^*(t)\geq 0$ for all $t\in [0,T].$
\end{enumerate}

\noindent  Note that $G^*G=I_{\K^k}$ and $F^*F=I_{\K^r}.$
The Hilbert space  $H:=L^2(0, 1; \mathbb K^n)$ is endowed with the standard $L_2-$norm $\|.\|_{L^2}.$ We define the realization $A(t)$ of $\A(t,\zeta,\partial)$ on $H$ by
\begin{equation}\label{pH-Operator-diss}
A(t)=-\frac{\partial}{\partial \zeta}\big(GS(t)\frac{\partial}{\partial \zeta}G^*+P_1\big)
  -P_0
\end{equation}
with domain
\begin{align}\nonumber D(A(t)):=&\left\{ u\in H: G^*u\in H^1(0,1;\K^k), \  GS(t)\frac{\partial }{\partial\zeta}G^*u+P_1u\in H^{1}(0,1; \mathbb K^n),\right.
   \\\nonumber&\qquad\qquad F^*\B_{\partial}(t)( u)=-W_R(t)F^*\mathcal C_{\partial}( u)\left. \hbox{ and }(I-FF^*)\mathcal C_{\partial}( u)=0
                                        \displaystyle  \right\}.
\end{align}
Thus the parabolic system  (\ref{dissipative system2})-(\ref{dissipative system2-BC2}) correspond to the non-autonomous abstract Cauchy problem
\begin{equation}\label{Non-aut-pbl-abstract example}
         \dot u(t) + A(t)\H(t)u(t) = f(t), \quad \text{a.e. \ on } \ [0,T],\ \ \qquad
                         \H(0)u(0)  =x_0.
    \end{equation}
We aim to investigate the well-posedness of $(\ref{Non-aut-pbl-abstract example})$
 with $L^p$-maximal regularity.
%

\subsection{Autonomous case}\label{Autonomous-general equation}
We consider in this subsection the autonomous case, i.e., that is  the parameters $\H(t)=\H,S(t)=S$ and $W_R(t)=W_R$ are independent of the time variable $t\in[0,T].$ Define the sesquilinear form $\fra:V\times V\to \K$  by
\begin{align}\label{form-example}
\fra(u,v):=(S(G^*u)'\mid (G^*v)')_{L_2}&+(P_1u\mid G(G^*v)')_{L_2}-([(I-GG^*)P_1u]'\mid v)_{L_2}
\\\nonumber &-(P_0u\mid v)_{L_2}+
     \mathcal C_{\partial}(v)^* FW_RF^*\mathcal C_{\partial}(u)
   \end{align}
with domain
\begin{equation}\label{form-example-domain}
 V:=\left\{v\in H: G^*v\in H^1(0,1;\K^k) \text{ such that }  (I-FF^*)\mathcal C_{\partial}(v)=0\right\}\end{equation}
where $V$ is equipped with the norm
\[\|v\|_{V}^2:= \|v\|_{L^2}^2+\|(G^*v)'\|_{L^2}^2.\]
The Hilbert space $V$ is continuously and  densely embedded into $H.$
\begin{lemma}\label{form-well define} The Hilbert space $V$ satisfies
\[V\subset \left\{v\in L^2(0,1;\K^n): (I-GG^*)P_1v\in H^1(0,1;\K^n)\right\}\]
and there exists $\kappa_2>0$ such that
\[\|(I-GG^*)P_1v\|_{H^1}\leq \kappa_2\|v\|_V.\]
In particular $\fra$ defined in (\ref{form-example})-(\ref{form-example-domain}) is well-defined.
\end{lemma}
\begin{proof}
By Assumption \ref{Assumption}.3 we may define $\zeta \mapsto R(\zeta)\in\K^{n\times n}$ by
\[ R(\zeta):=\left\{%
\begin{array}{ll}
    (I-GG^*)P_1(\zeta) &\hbox{ on }   {\rm ran }\ GG^*
    ,\\
    0 & \hbox{ on  } ({\rm ran }\ GG^*)^\bot={\rm ker }\ GG^*, \\
\end{array}%
\right. \]
 which implies  $R\in W^{1,\infty}(0,1;\K^{n\times n}).$ From here the assertion is immediate.
\end{proof}
\begin{lemma}\label{associate form example} The sesquilinear form $\fra:[0,T]\times V\times V\rightarrow \K$ defined by (\ref{form-example})-(\ref{form-example-domain}) is continuous and $H$-elliptic.
\end{lemma}
\begin{proof} We may and will assume that $P_0=0.$
Let $v\in  V.$ Let $\varepsilon>0$ such that $m_2-1/{2\varepsilon}>0.$
It follows from Assumption \ref{Assumption}.5, Assumption \ref{Assumption}.3 and Young's inequality
\begin{align*}\Re\fra(v,v)
&=\Re (S(G^*v)'\mid (G^*v)')_{L_2}
-\Re(P_1v\mid G(G^*v)')_{L_2}
\\&\hskip5cm -\Re([(I-GG^*)P_1v]'\mid v)_{L^2}
\\&\geq  m_2\|(G^*v)'\|_{L^2}^2-\|P_1\|_\infty\|G\|\|v\|_{L^2}\|(G^*v)'\|_{L^2}
\\&\hskip5cm -\|[(I-GG^*)P_1v]'\|_{L^2}\|v\|_{L^2}
\\&\geq  m_2\|(G^*v)'\|_{L^2}^2-\|P_1\|_\infty\|G\|\|v\|_{L^2}\|(G^*v)'\|_{L^2}
\\&\qquad-\kappa_2(\|(G^*v)'\|_{L^2}+\|v\|_{L^2})\|v\|_{L^2}
\\&= m_2\|(G^*v)'\|_{L^2}^2 -(\|P_1\|_\infty\|G\|+\kappa_2)\|(G^*v)'\|_{L^2}
\|v\|_{L^2}-\kappa_2\|v\|_{L^2}^2
\\&\geq (m_2-1/{\varepsilon})\|(G^*v)'\|_{L^2}^2-(\frac{\varepsilon}{2}\tilde\kappa^2_2+\kappa_2)\|v\|_{L^2}^2
\end{align*}
where $\tilde\kappa_2:=\|P_1\|_\infty\|G\|+\kappa_2.$ Thus  \[\Re\fra(v,v)+\omega\|v\|_{L^2}^2\geq \alpha\| v\|_{V}^2,\]
where $\omega:=1+\frac{\varepsilon}{2}\tilde\kappa^2+\kappa_2$ and $\alpha:=\min\{1, (m_2-1/{\varepsilon})\}.$
The continuity follows easily from Lemma \ref{form-well define}, the Cauchy-Schwartz inequality and the Sobolev embedding Theorem.
\end{proof}
\noindent We define on $H$ the operator
\begin{equation}\label{pH-Operator-diss}
Au=-\frac{\partial}{\partial \zeta}\big(GS\frac{\partial}{\partial \zeta}G^*u+P_1u\big)
  -P_0u
\end{equation}
with domain
\begin{align}\label{pH-Operator-diss-Dom} D(A):=&\left\{ u\in H: G^*u\in H^1(0,1;\K^k), \  GS\frac{\partial }{\partial\zeta}G^*u+P_1u\in H^{1}(0,1; \mathbb K^n),\right.
   \\\nonumber&\qquad\qquad F^*\B_{\partial}( u)=-W_RF^*\mathcal C_{\partial}( u)\left. \hbox{ and }(I-FF^*)\mathcal C_{\partial}( u)=0
                                        \displaystyle  \right\}
\end{align}
\begin{proposition}\label{Thm-gene-operateor-asso}
The operator associated with $\fra$ on $H$ is the operator $(A,D(A))$ defined by (\ref{pH-Operator-diss})-(\ref{pH-Operator-diss-Dom}), and thus $-A$ generates a holomorphic $C_0-$semigroup.
\end{proposition}
 \begin{proof}Without loss of generality, we may and will assume $P_0=0.$
Denote by $(B,D(B))$ the operator associated with $\fra$ on $H,$ i.e.,
\begin{align*}
D(B):=&\{ u\in  V: \exists f\in H \hbox{ such that }  \mathfrak a(u,\psi)=(f\mid \psi) \hbox{ for all  } \psi\in V\}
\\ Bu:=&f.
\end{align*}
Let $u\in D(A)$. Then for all $v\in V$ we have
\begin{align}
(Au\mid v)_{L^2}\nonumber&=-\big(\frac{\partial}{\partial \zeta}\big(GS\frac{\partial}{\partial \zeta}G^*u+P_1u\big)\mid v\big)_{L^2}
\\\nonumber&=-\big((I-GG^*)\frac{\partial}{\partial \zeta}\big(GS\frac{\partial}{\partial \zeta}G^*u+P_1u\big)\mid v\big)_{L^2}
\\\nonumber&\hskip2.5cm-\big(\frac{\partial}{\partial \zeta}\big(GS\frac{\partial}{\partial \zeta}G^*u+P_1u\big)\mid GG^* v\big)_{L^2}
\\\nonumber&=-\big(\frac{\partial}{\partial \zeta}(I-GG^*)P_1u\mid v\big)_{L^2}+\big(GS\frac{\partial}{\partial \zeta}G^*u+P_1u\mid G\frac{\partial}{\partial \zeta}G^* v\big)_{L^2}
\\\label{Boundary-Con}&\hskip2cm+\Big[(GG^* v)(\zeta)^*(-GS\frac{\partial}{\partial \zeta}G^*u-P_1u\big)(\zeta)\Big]_0^1
\end{align}
Here we have used the fact that \[(I-GG^*)\big(GS\frac{\partial}{\partial \zeta}G^*u+P_1u\big)=
(I-GG^*)P_1u.\]
The condition  \[(I-FF^*)\mathcal C_\partial(v)=0\]
in the definition of $V$ and  the fact that $u\in D(A)$ imply that the boundary term in (\ref{Boundary-Con}) is equal to
\begin{align*}
\mathcal C_\partial(v)^*\mathcal B_\partial(u)&=\mathcal C_\partial(v)^*(I-FF^*+FF^*)\mathcal B_\partial(u)
\\&=\mathcal C_\partial(v)^*FW_RF^*\mathcal C_\partial(u).
\end{align*}
Thus $\fra(u,v)=(Au\mid v)_{L^2}.$ This proves $A\subset B.$
For the converse inclusion, let $u\in D(B).$ Then \begin{align}\label{App-eq0}
(Bu\mid v)_{L^2}&=\fra(u,v)
\\\nonumber&=(S(G^*u)'\mid (G^*v)')_{L_2}+(P_1u\mid G(G^*v)')_{L_2}-([(I-GG^*)P_1u]'\mid v)_{L_2}
\\\nonumber&=(GS(G^*u)'\mid v')_{L_2}+(P_1u\mid GG^*v')_{L_2}+((I-GG^*)P_1u\mid v')_{L_2}
\\\nonumber&=(GS(G^*u)'+P_1u\mid v')_{L_2}
\end{align}
for all $v\in C_c^\infty(0,1;\K^n)\subset V.$ This means, by the definition of the weak derivative, that $GS(G^*u)'+P_1u\in H^1(0,1,\K^n)$  and
\begin{equation}\label{weak-der}
Bu=-\frac{\partial }{\partial\zeta}(GS\frac{\partial }{\partial\zeta}(G^*u)+P_1u)
\end{equation}
Let $v\in  V$. Inserting (\ref{weak-der}) in (\ref{App-eq0}) and integrating by part we obtain
\begin{align*} &\mathcal C_\partial(v)^*FW_RF^*\mathcal C_\partial(u)=-\mathcal C_\partial(v)^*FF^*\mathcal B_\partial(u).
\end{align*}
On the other hand, for each $z\in\K^r$ there exists $v\in  V$ such that
\[z=F^*\mathcal C_\partial(v).\]
In fact, remark that $\{\mathcal C_\partial(v), v\in V \}=\ker(I-FF^*)$
and $\K^r={\rm ran}\ F^*F=F^*{\rm ran}\ F=F^*\ker(I-FF^*).$ We conclude that
\[F^*\mathcal B_\partial(u)=-W_RF^*\mathcal C_\partial(u).\]
Therefore, $u\in D(A)$ and $Bu=Au.$ This completes the proof.
\end{proof}

\par\noindent Now Proposition \ref{generation thm example} below follows from Lemma \ref{associate form example}, Proposition \ref{Thm-gene-operateor-asso}  and Proposition \ref{Mult-pert-Auto}.
\begin{proposition}\label{generation thm example}
 The operator  $-A\mathcal H$ given by \begin{equation*}\label{Pert-General-port-Hamiltonian-system}
 A\mathcal H:=-\frac{\partial}{\partial \zeta}\big(GS\frac{\partial}{\partial \zeta}G^*\mathcal H+P_1\mathcal H\big)
  -P_0\mathcal H
 \end{equation*}
with domain
\[D(A\mathcal H)=\{u\in H \ : \mathcal H u\in D(A)\}\]
generates  a holomorphic $C_0$-semigroup on $H.$
 \end{proposition}
Next, Proposition \ref{contraction semigroup example} below gives additional conditions under which the $-A\H$ generates a contraction semigroup. 
 \begin{proposition}\label{contraction semigroup example} Assume that the following assumptions holds.
 \begin{description}
 \item[$(i)$] $W_R+F^*\left(
                       \begin{array}{cc}
                         G^*P_1(1)G & 0 \\
                         0 & -G^*P_1(0)G \\
                       \end{array}
                     \right)F^*\geq 0$
  \item[$(ii)$] $\Re(P_0(.)+GG^*P_1'(.)+\frac{1}{2}GG^*P'(.)GG^*)\leq 0 $
      \item[$(iii)$] $P_1(.)=P_1(.)^*$
  \end{description}
  Then $-A\H$ generates a contractive semigroup on $H$  with respect to the scalar product
  \[(u\mid v)_\H:=(\H u\mid v)_{L^2}.\]
 \end{proposition}
 \begin{proof} It suffices  to prove that the sesquilinear $(\fra,V)$ given by (\ref{form-example})-(\ref{form-example-domain})
 is accretive, i.e., $\Re\fra(u,u)\geq 0$ for all $u\in V.$ In fact, $-A$ generates a contractive semigroup on $H$ if and only if $-A\H$ generates a contractive semigroup on $(H, (.\mid .)_\H).$
 \par\noindent From Assumption \ref{Assumption} we deduce that $(I-GG^*)P_1=(I-GG^*)P_1GG^*.$ Thus for each $u\in V$
 \begin{align*}
 \Re(u\mid &[(P_1GG^*-(I-GG^*)P_1)u]')_{L^2}=\Re(u\mid [(P_1-(I-GG^*)P_1)GG^*u]')_{L^2}
 \\&=\Re(u\mid GG^*(P_1GG^*u)')_{L^2}=\Re(GG^*u\mid (P_1GG^*u)')_{L^2}
 \\&=-\frac{1}{2}(GG^*u\mid P'_1(GG^*u))_{L^2}+\frac{1}{2}[(GG^*u)^*(\zeta)P_1(\zeta)(GG^*u)(\zeta)]_0^1
 \\&=-\frac{1}{2}(GG^*u\mid P'_1(GG^*u))_{L^2}+
 \\&\frac{1}{2}\left(
     \begin{array}{cc}
       (G^*u)(1)   \\
       (G^*u)(0)   \\
     \end{array} \right)^*FF^*\left(
                       \begin{array}{cc}
                         G^*P_1(1)G & 0 \\
                         0 & -G^*P_1(0)G^* \\
                       \end{array}
                     \right)FF^*\left(
     \begin{array}{cc}
       (G^*u)(1)   \\
       (G^*u)(0)   \\
     \end{array} \right).
\end{align*}
It follows from $(i)-(iii)$ and (\ref{form-example}) that $\fra$ is accretive. This is equivalent to the fact that $-A$ generates a contraction semigroup.
 \end{proof}
%
\subsection{Non-autonomous case}\label{example-nonautonomous}
Let us come back to the non-autonomous situation and recall Assumption \ref{Assumption}. We observe $L^p$-maximal regularity in the following two cases.
\par \textbf{\textit{$1^{st}$ case}}: Let $p\in(1,\infty)$ be arbitrary. We then assume that $S$ and $W_R$ do not depend on the time variable $t\in [0,T]$ and obtain the   following well-posedness result.
\begin{theorem}\label{Lp-MR-port-H} Let $p\in(1,\infty)$ and assume that $S(t,.)=S(.), W_R(t)=W_R$ do not depend on $t\in[0,T].$ Then for any given $x_0\in (H,D(A))_{1-\frac{1}{p},p}$ and $f\in L^p(0,T;H)$ there exists a unique $u\in MR_\H(p,H)$ satisfying
the non-autonomous  system
\begin{align}\nonumber
  \partial_t u(t,\zeta)+\A(\zeta,\partial)\H(t,\zeta)u(t,\zeta)&=f(t,\zeta)
\\ \nonumber \H(0,\zeta)u(0,\zeta)&=x_0(\zeta)
\\ \nonumber F^*\B_{\partial}(\H(t) u)&=-W_RF^*\mathcal C_{\partial}(\H(t) u)
\\ \nonumber (I-FF^*)\mathcal C_{\partial}(\H(t) u)&=0.
  \end{align}

\end{theorem}
\begin{proof} By Assumption \ref{Assumption}.4 $t\mapsto\H(t,.)$ is Lipschitz continuous and uniformly positive as a function $[0,T]\longrightarrow \L(H).$ Moreover, $A(t)=A$ is constant, so the result follows from Theorem \ref{per-thm-Banach}.
\end{proof}
\par \textbf{\textit{$2^{nd}$ case}}: $p=2.$ In this case we do not impose additional assumption on $S, W_R,$ besides Assumption A.
\begin{theorem}\label{L2-Mr-port-H} Given $x_0\in V$ and $f\in L^2(0,T;H)$
the non-autonomous  system
 \begin{align}\nonumber
  \partial_t u(t,\zeta)+\A(t,\zeta,\partial)\H(t,\zeta)u(t,\zeta)&=f(t,\zeta)
\\ \nonumber \H(0,\zeta)u(0,\zeta)&=x_0(\zeta)
\\ \nonumber F^*\B_{\partial}(t)(\H(t) u)&=-W_R(t)F^*\mathcal C_{\partial}(\H(t) u)
\\ \nonumber (I-FF^*)\mathcal C_{\partial}(\H(t) u)&=0
  \end{align}
has a unique solution $u\in H^1(0,T;H)\cap L^2(0,T;V).$
\end{theorem}
\begin{proof} The result follows from Theorem \ref{Multiplicative-perturbation-nonautonomous-Form} for
\[
\fra_1(t,u,v):=(S(t)(G^*u)'\mid (G^*v)')_{L_2}+\B_{\partial}(t)(v)^*FW_R(t)F^*
     \B_{\partial}(t)(u)\]
and
\[\fra_2(t,u,v)=(P_1u\mid G(G^*v)')_{L_2}-([(I-GG^*)P_1u]'\mid v)_{L_2}-(P_0u\mid v)_{L_2}.\]
\end{proof}
%
\subsection{Wave equation with structural damping}
We illustrate our theoretical results of Section \ref{Example: general class} and \ref{Autonomous-general equation}  by means of the one-dimensional wave equation with structural damping along the spatial domain.
We start with the autonomous and homogeneous evolution equation
\begin{equation}\label{wave equation}\rho(\zeta)\frac{\partial^2\omega}{\partial t^2}=\frac{\partial}{\partial\zeta}\Big(E(\zeta)\frac{\partial\omega}{\partial\zeta}(t,\zeta)\Big)
+\frac{\partial}{\partial\zeta}\Big(k(\zeta)\frac{\partial^2\omega}{\partial\zeta\partial t}(t,\zeta)\Big)
\end{equation}
where $\zeta\in[0,1]$ is the spatial variable, $\omega(t,\zeta)$ is the deflection at point $\zeta$ and time $t,$ $\rho(.)$ is the mass density, $T(.)$ is the Young's modulus and $k(.)$ is the damping coefficient. The distributed parameters $T,\rho,k$ are assumed to be of class  $L^{\infty}(0,1)$  and strictly positive with
\[\partial \leq\rho(\zeta),E(\zeta),k(\zeta) \ \ \zeta\ a.e.\text{ for some } \partial>0.\]
We define $x_1:=\rho\frac{\partial\omega}{\partial t}$ and  $x_2:=\frac{\partial\omega}{\partial \zeta}.$ Then
(\ref{wave equation}) can be equivalently  written as
\begin{align*}
\frac{\partial}{\partial t}\left(
                        \begin{array}{cc}
                          x_1  \\
                           x_2 \\
                        \end{array}
                      \right)&=\frac{\partial}{\partial \zeta}\Big[\left(
                        \begin{array}{cc}
                          1  \\
                          0 \\
                        \end{array}
                      \right)k(\zeta)\frac{\partial}{\partial \zeta}\left(
                        \begin{array}{cc}
                          1 & 0 \\
                        \end{array}
                      \right)\left(
                        \begin{array}{cc}
                          1/\rho(\zeta) & 0 \\
                          0  & E(\zeta)   \\
                        \end{array}
                      \right)\left(
                        \begin{array}{cc}
                          x_1  \\
                           x_2 \\
                        \end{array}
                      \right)\\&\hskip 5cm +\left(
                        \begin{array}{cc}
                          E(\zeta)x_2  \\
                           x_1/\rho(\zeta) \\
                        \end{array}
                      \right)\Big]
\\&=\frac{\partial}{\partial \zeta}\Big[\left(
                        \begin{array}{cc}
                          1  \\
                          0 \\
                        \end{array}
                      \right)k(\zeta)\frac{\partial}{\partial \zeta}\left(
                        \begin{array}{cc}
                          1 & 0 \\
                        \end{array}
                      \right)\left(
                        \begin{array}{cc}
                          1/\rho(\zeta) & 0 \\
                          0  & E(\zeta)   \\
                        \end{array}
                      \right)\left(
                        \begin{array}{cc}
                          x_1  \\
                           x_2 \\
                        \end{array}
                      \right)\\& \hskip3cm+\left(
                        \begin{array}{cc}
                          0 & 1 \\
                          1 & 0 \\
                        \end{array}
                      \right)\left(
                        \begin{array}{cc}
                          1/\rho(\zeta) & 0 \\
                          0  & E(\zeta)   \\
                        \end{array}
                      \right)\left(
                        \begin{array}{cc}
                          x_1  \\
                           x_2 \\
                        \end{array}
                      \right)\Big].
\end{align*}
We see  that the damped wave equation (\ref{wave equation}) can be written in the form (\ref{dissipative system2}) with
\[\mathcal H(\zeta)=\left(
                        \begin{array}{cc}
                          \frac{1}{\rho(\zeta)} & 0 \\
                           0 & E(\zeta) \\
                        \end{array}
                      \right),
                      \ P_1(\zeta)=\left(
                        \begin{array}{cc}
                          0& 1 \\
                           1 & 0 \\
                        \end{array}
                      \right),\ G=\left(
                        \begin{array}{cc}
                          1  \\
                          0 \\
                        \end{array}
                      \right),\]
$S(\zeta)=k(\zeta),\ P_0=0,\ f = 0,\ n = 2$ and  $k = 1.$  Furthermore, it is easy to see that the $G,P_0,P_1$ satisfy Assumption \ref{Assumption}. In particular, we have
	\begin{equation}
	 (I - GG^*)P_1
	  = \left( \begin{array}{cc} 0& 0 \\0 &1 \end{array} \right) \left( \begin{array}{cc} 0&1 \\ 1&0 \end{array} \right)
	  = \left( \begin{array}{cc} 0&0 \\ 1&0 \end{array} \right)
	  = \left( \begin{array}{cc}0 &1 \\ 1&0 \end{array} \right) GG^*
	  \nonumber
	\end{equation}
so that equation (\ref{couplung-cond}) is satisfied.
Also note that $\H$ and $S$ are coercive multiplication operators on $L^2(0,1;\K^2)$ and $L^2(0,1;\K)$, respectively. Moreover, remark that the assumptions of Proposition \ref{contraction semigroup example} are satisfied. So far we did not impose any boundary conditions. First consider the essential boundary conditions, choosing $r \in \{ 0, 1, 2\}$ and $F \in \K^{2 \times r}$ such that $FF^* \in \K^{2 \times 2}$ is a projection.
We then set $V := H^1_F(0,1) \times L^2(0,1)$ where
	\begin{equation}
	 H^1_F(0,1)
	  := \left\{ v \in H^1(0,1): (I-FF^*) \left( \begin{array}{c} v(0) \\ v(1) \end{array} \right)=0 \right\}.
	  \nonumber
	\end{equation}
We give some examples.
	\begin{enumerate}
	 \item
	  $r = 0$, then $FF^* = 0$ and this leads to
		\begin{equation}
		 V
		  = \{ v \in L_2(0,1;\K^2): \ \frac{v_1}{\rho} \in H^1(0,1), \ \frac{v_1}{\rho}(0) = \frac{v_1}{\rho}(1) = 0 \}
		  \nonumber
		\end{equation}
	  i.e., Dirichlet boundary conditions $\omega_t(0) = \omega_t(1) = 0$.
	 \item
	  $r = 2$, then $FF^* = I$ and consequently
		\begin{equation}
		 V
		  = \{ v \in L_2(0,1;\K^2): \ \frac{v_1}{\rho} \in H^1(0,1) \}
		  \nonumber
		\end{equation}
	  i.e., no essential boundary conditions.
	 \item
	  $r = 1$, e.g., $F = \frac{1}{2} \left( \begin{array}{c} 1 \\ 1 \end{array} \right)$, then
		\begin{equation}
		 V
		  = \{ v \in L_2(0,1;\K^2): \ \frac{v_1}{\rho} \in H^1(0,1), \ \frac{v_1}{\rho}(0) = \frac{v_1}{\rho}(1) \}
		  \nonumber
		\end{equation}
	  i.e., periodic boundary conditions $\omega_t(0) = \omega_t(1)$.
	\end{enumerate}
Secondly, by choosing $0 \leq W_R = W_R^* \in \K^{r \times r}$ we demand natural boundary conditions
	\begin{equation}
	 F^* \left( \begin{array}{c} (k(.) (\frac{u_1}{\rho})' + E(.) u_2)(1) \\ -(k(.) (\frac{u_1}{\rho})' + E(.) u_2)(0) \end{array} \right)
	  = - W_R F^* \left( \begin{array}{c} \frac{u_1}{\rho}(1) \\ \frac{u_1}{\rho}(0) \end{array} \right)
	  \nonumber
	\end{equation}
which for the original equation (\ref{wave equation}) correspond to the boundary conditions
	\begin{equation}
	 F^* \left( \begin{array}{c} (k(.) \omega_{t\z} + E(.) \omega_\z)(1) \\ -(k(.) \omega_{t\z} + E(.) \omega_\z)(0) \end{array} \right)
	  = - W_R F^* \left( \begin{array}{c} \omega_t(1) \\ \omega_t(0) \end{array} \right).
	  \nonumber
	\end{equation}
For our previous examples this reads
	\begin{enumerate}
	 \item
	  $r = 0$, then we have only essential boundary conditions.
	 \item
	  $r = 2$, then for $W_R = 0$ we obtain Neumann-type boundary conditions
		\begin{equation}
		 (k(.) \omega_{t\z} + E(.) \omega_\z)(0)
		  = (k(.) \omega_{t\z} + E(.) \omega_\z)(1)
		  = 0
		  \nonumber
		\end{equation}
	  and for $0 \not= W_R \geq 0$ Robin-type boundary conditions.
	 \item
	  $r = 1, \ F = \frac{1}{2} \left( \begin{array}{c} 1 \\ 1 \end{array} \right)$, then $W_R \geq 0$ is scalar and the natural boundary condition reads
		\begin{equation}
		 (k(.)\omega_{t\z} + E(.)\omega_\z)(1) - (k(.)\omega_{t\z} + E(.)\omega_\z)(0)
		  = - W_R (\omega_t(1) + \omega_t(0)).
		  \nonumber
		\end{equation}
	\end{enumerate}
For the non-autonomous version of equation (\ref{wave equation}) the results of Section \ref{example-nonautonomous} thus imply the following where we use the notation
	\begin{align}
	 D_{F,W_R,k}
	  &:= \{ u \in H^1_F(0,1) \times L_2(0,1): (k u_1'+u_2) \in H^1(0,1),
	  \nonumber \\
	  &\qquad \left( \begin{array}{cc} (k u_1'+u_2)(0) \\ - (k u_1'+u_2)(1) \end{array} \right) = - W_R \left( \begin{array}{c} u_1(0) \\ u_1(1) \end{array} \right) \}.
	  \nonumber
	\end{align}

\begin{proposition}
Let $p \in (1, \infty)$ and assume that $\rho, T: [0,T] \times [0,1] \rightarrow \R, \ k: [0,1] \rightarrow \R$ are bounded and measurable with $\rho$ and $T$ piecewise Lipschitz continuous in $t \in [0,T]$ and such that
	\begin{equation}
	 \rho(t,\z), T(t,\z), k(\z)
	  \geq \partial
	  > 0,
	  \quad
	  \text{a.e.} \ (t,\z) \in [0,T] \times [0,1].
	  \nonumber
	\end{equation}
Further let $r \in \{0, 1, 2\}$ and $F \in \K^{2 \times r}$ such that $FF^* \in \K^{2 \times 2}$ is a projection and $W_R = W_R^* \geq 0$ a $r \times r$-matrix.
Then for every $(x_1, x_2) \in (L_2(0,1;\K^2),D_{F,W_R,k})_{1/p^*,p}$ and $f \in L^p(0,T;L^2(0,1))$ the problem
	\begin{align}
	 \rho(t,\z) \frac{\partial^2 \omega}{\partial t^2}
	  - \frac{\partial}{\partial \z} \left( T(t,\z) \frac{\partial \omega}{\partial \z}(t,\z) - k(\z) \frac{\partial^2 \omega}{\partial t \partial \z}(t,\z) \right)
	  &= f(t,\z)
	  \nonumber \\
	 (I-FF^*) \left( \begin{array}{c} \omega_t(0) \\ \omega_t(1) \end{array} \right)
	  &= 0
	  \nonumber \\
	 F^* \left( \begin{array}{c} (k(.) \omega_{t\z} + E(t,.) \omega_\z)(1) \\ -(k(.) \omega_{t\z} + E(t,.) \omega_\z)(0) \end{array} \right)
	  &= - W_R F^* \left( \begin{array}{c} \omega_t(1) \\ \omega_t(0) \end{array} \right)
	  \nonumber \\
	 (E\omega_t)(0,.)
	  &= x_1
	  \nonumber \\
	 (E\omega_\z)(0,.)
	  &= x_2
	  \nonumber
	\end{align}
has a solution $\omega$ such that
	\begin{align}
	 (\omega_t, E \omega_\z)
	  &\in W^{1,p}(0,T;L^2(0,1;\C^2))
	  \nonumber \\
	 k(.) \omega_t + E(t,.) \omega_\z
	  &\in L^p(0,T;H^1(0,1;\C^2))
	  \nonumber
	\end{align}
which is unique up to an additive constant $\Delta  \in \K$.
\end{proposition}

\begin{proof}
The result follows from the previous considerations and Theorem \ref{Lp-MR-port-H}.
\end{proof}

\begin{proposition}
Let $p = 2$ and assume that $\rho, E, k: [0,T] \times [0,1] \rightarrow \R$ are bounded and measurable and piecewise Lipschitz continuous in $t \in [0,T]$ and such that
	\begin{equation}
	 \rho(t,\z), E(t,\z), k(t,\z)
	  \geq \delta
	  > 0,
	  \quad
	  \text{a.e.} \ (t,\z) \in [0,T] \times [0,1].
	  \nonumber
	\end{equation}
Further let $r \in \{0, 1, 2\}$ and $F \in \K^{2 \times r}$ such that $FF^* \in \K^{2 \times 2}$ is a projection and $W_R: [0,T] \rightarrow \K^{r \times r}$ piecewise Lipschitz continuous with $W_R(t) = W_R^*(t) \geq 0$ for all $t \in [0,T]$.
Then for every $(x_1, x_2) \in H^1_F(0,1) \times L_2(0,1)$ and $f \in L^2(0,T;L^2(0,1))$ the problem
	\begin{align}
	 \rho(t,\z) \frac{\partial^2 \omega}{\partial t^2}
	  - \frac{\partial}{\partial \z} \left( T(t,\z) \frac{\partial \omega}{\partial \z}(t,\z) - k(t,\z) \frac{\partial^2 \omega}{\partial t \partial \z}(t,\z) \right)
	  &= f(t,\z)
	  \nonumber \\
	 (I-FF^*) \left( \begin{array}{c} \omega_t(0) \\ \omega_t(1) \end{array} \right)
	  &= 0
	  \nonumber \\
	 F^* \left( \begin{array}{c} (k(t,.) \omega_{t\z} + T(t,.) \omega_\z)(1) \\ -(k(t,.) \omega_{t\z} + T(t,.) \omega_\z)(0) \end{array} \right)
	  &= - W_R(t) F^* \left( \begin{array}{c} \omega_t(1) \\ \omega_t(0) \end{array} \right)
	  \nonumber \\
	 (E\omega_t)(0,.)
	  &= x_1
	  \nonumber \\
	 (E\omega_\z)(0,.)
	  &= x_2
	  \nonumber
	\end{align}
has a solution $\omega$ such that
	\begin{align}
	 (\omega_t, E \omega_\z)
	  &\in H^1(0,T;L^2(0,1;\K^2))
	  \nonumber \\
	 k(t,.) \omega_t + E(t,.) \omega_\z
	  &\in L^2(0,T;H^1(0,1;\K^2))
	  \nonumber
	\end{align}
which is unique up to an additive constant $\Delta \in \K$.
\end{proposition}
\begin{proof}
This is a consequence of Theorem \ref{L2-Mr-port-H}.
\end{proof}

\bigskip
{\em Bj\"orn Augner, Birgit Jacob, Hafida Laasri}: University of wuppertal, Work Group Functional Analysis,
42097 Wuppertal, Germany.
\\ augner@uni-wuppertal.de, jacob@math.uni-wuppertal.de, laasri@uni-wuppertal.de;

\end{document}